\documentclass[12pt,a4paper]{amsart}
\usepackage{amssymb,amscd}

\theoremstyle{plain}
\newtheorem{theorem}{Theorem}[section]
\newtheorem{lemma}[theorem]{Lemma}
\newtheorem{proposition}[theorem]{Proposition}

\theoremstyle{definition}

\newtheorem{remark}[theorem]{Remark}
\newtheorem{remarks}[theorem]{Remarks}
\newtheorem{example}[theorem]{Example}

\numberwithin{equation}{section}

\newcommand\bC{{\mathbb C}}
\newcommand\bF{{\mathbb F}}
\newcommand\bG{{\mathbb G}}

\newcommand\bQ{{\mathbb Q}}

\newcommand\bV{{\mathbb V}}
\newcommand\bZ{{\mathbb Z}}

\newcommand\cL{{\mathcal L}}
\newcommand\cM{{\mathcal M}}

\newcommand\cO{{\mathcal O}}

\newcommand\kb{\bar{k}}

\newcommand\fm{{\mathfrak m}}

\newcommand\ab{\operatorname{ab}}

\newcommand\aff{\operatorname{aff}}
\newcommand\ant{\operatorname{ant}}

\newcommand\gug{\operatorname{gug}}
\newcommand\id{\operatorname{id}}

\newcommand\tor{\operatorname{tor}}

\newcommand\Aut{\operatorname{Aut}}
\newcommand\Hom{\operatorname{Hom}}

\newcommand\Pic{\operatorname{Pic}}
\newcommand\Spec{\operatorname{Spec}}
\newcommand\Supp{\operatorname{Supp}}

\title{Anti-affine algebraic groups}

\author{Michel Brion}
\address{Universit\'e de Grenoble I\\
D\'epartement de Math\'ematiques\\
Institut Fourier, UMR 5582 du CNRS\\
38402 Saint-Martin d'H\`eres Cedex, France}
\email{Michel.Brion@ujf-grenoble.fr}

\begin{document}
 
\begin{abstract}
We say that an algebraic group $G$ over a field is anti-affine 
if every regular function on $G$ is constant. We obtain a
classification of these groups, with applications to the structure 
of algebraic groups in positive characteristics, and to the
construction of many counterexamples to Hilbert's fourteenth 
problem.
\end{abstract}

\maketitle

\section{Introduction}
\label{sec:introduction}

We say that a group scheme $G$ of finite type over a field $k$ is
anti-affine if $\cO(G)=k$; then $G$ is known to be smooth, connected
and commutative. Examples include abelian varieties, their universal
vector extensions (in characteristic $0$ only) and certain
semi-abelian varieties.

\smallskip

The classes of anti-affine groups and of affine (or, 
equivalently, linear) group schemes play complementary roles in 
the structure of group schemes over fields. Indeed, any connected 
group scheme $G$, of finite type over $k$, has a smallest normal
subgroup scheme $G_{\ant}$ such that the quotient $G/G_{\ant}$ is affine. 
Moreover, $G_{\ant}$ is anti-affine and central in $G$ (see \cite{DG70}). 
Also, $G$ has a smallest normal connected affine subgroup 
scheme $G_{\aff}$ such that $G/G_{\aff}$ is an abelian variety 
(as follows from Chevalley's structure theorem, see 
\cite{BLR90}). This yields the \emph{Rosenlicht decomposition}: 
$G = G_{\aff} \, G_{\ant}$ and $G_{\aff} \cap G_{\ant}$ contains 
$(G_{\ant})_{\aff}$ as an algebraic subgroup of finite index
(see \cite{Ro56}).

\smallskip

Affine group schemes have been extensively investigated, but 
little seems to be known about their anti-affine counterparts; 
they only appear implicitly in work of Rosenlicht and Serre 
(see \cite{Ro58, Ro61, Se58a}). Here we obtain some fundamental 
properties of anti-affine groups, and reduce their structure to
that of abelian varieties.

\smallskip

In Theorem \ref{thm:cla}, we classify anti-affine 
algebraic groups $G$ over an arbitrary field $k$ with separable
closure $k_s$. In positive characteristics, $G$ is a semi-abelian
variety, parametrized by a pair $(A,\Lambda)$ where $A$ is an abelian 
variety over $k$, and $\Lambda$ is a sublattice of $A(k_s)$, stable
under the action of the Galois group.  
The classification is a bit more complicated in characteristic 
$0$: the parameters are then triples $(A,\Lambda,V)$
where $A$ and $\Lambda$ are as above, and $V$ is a subspace of 
the Lie algebra of $A$. In both cases, $A$ is the dual of the 
abelian variety $G/G_{\aff}$.

\smallskip

We illustrate this classification by describing the universal
morphisms from anti-affine varieties to commutative algebraic groups,
as introduced by Serre (see \cite{Se58a,Se58b}).

\smallskip

Together with the Rosenlicht decomposition, our classification
yields structure results for several classes of group schemes.
As a first consequence, any connected commutative group scheme
over a perfect field $k$ is the almost direct product 
of an anti-affine group, a torus, and a connected unipotent
group scheme (see Theorem \ref{thm:com} for a precise statement).

\smallskip

In another direction, if the ground field $k$ is finite, then any 
anti-affine group is an abelian variety. This gives back a remarkable
result of Arima: any connected group scheme over a finite field has
the decomposition $G = G_{\aff} \, G_{\ab}$ where $G_{\ab}$ is the
largest abelian subvariety of $G$; moreover, $G_{\aff} \cap G_{\ab}$
is finite (see \cite{Ar60,Ro61}).

\smallskip

Arima's result does not extend to (say) uncountable and 
algebraically closed fields, as there exist semi-abelian varieties 
that are anti-affine but non-complete. Yet we obtain a 
structure result for connected algebraic groups over perfect 
fields of positive characteristics, namely, the decomposition 
$G = H \, S$ where $H \subset G_{\aff}$ denotes 
the smallest normal connected subgroup such that 
$G_{\aff}/H$ is a torus, and $S \subset G$ is a 
semi-abelian subvariety; moreover, $H \cap S$ is finite 
(Theorem \ref{thm:dec}). 

\smallskip

Our classification also has rather unexpected applications to 
Hilbert's fourteenth problem. In its algebro-geometric formulation, 
it asks if the coordinate ring of every quasi-affine variety is 
finitely generated (see \cite{Za54}, and \cite{Win03} for the
equivalence with the invariant-theoretic formulation). The answer is
known to be negative, the first counterexample being due to Rees (see
\cite{Re58}). Here we obtain many counterexamples, namely, all
$\bG_m$-torsors associated to ample line bundles over anti-affine,
non-complete algebraic groups (Theorem \ref{thm:zar}). 

\smallskip

Some of the preceding statements bear a close analogy 
to known results on complex Lie groups. Specifically, any 
connected complex Lie group $G$ has a smallest closed normal 
subgroup $G_{\tor}$ such that the quotient $G/G_{\tor}$ 
is Stein. Moreover, $G_{\tor}$ is connected and central in $G$, 
and every holomorphic function on $G_{\tor}$ is constant. The latter 
property defines the class of toroidal complex Lie groups, 
also known as Cousin groups, or quasi-tori, or (H.C) groups.
Toroidal groups may be parametrized by pairs $(T, \Lambda)$ where 
$T$ is a complex torus, and $\Lambda$ is a sublattice of the dual 
torus. Any connected commutative complex Lie group admits a unique 
decomposition $G = G_{\tor} \times (\bC^*)^m \times \bC^n$ 
(see the survey \cite{AK01} for these results). 
Yet this analogy is incomplete, as Chevalley's structure theorem
admits no direct generalization to the setting of complex Lie
groups. In fact, the maximal closed connected Stein subgroups of a
connected Lie group need not be pairwise isomorphic (see \cite{AK01}
again), or normal, or co-compact.

\smallskip

Returning to the algebraic setting, our structure results have
applications to homogeneous spaces, which will be developed
elsewhere. A natural question asks for their generalizations to group
schemes over (say) local artinian rings, or discrete valuation rings.

\bigskip

\noindent
{\bf Acknowledgements.}
Most of the results of this article were first presented at the 
2007 Algebra Summer School in Edmonton. I thank the organizers 
and the participants for their stimulating interest.
Also, many thanks are due to St\'ephane Druel, Adrien Dubouloz, David
Harari and Ga\"el R\'emond for fruitful discussions. 

\smallskip

After a first version of this text was posted on arXiv, I was informed
by Carlos Sancho de Salas of his much earlier book \cite{Sa01}, 
where the classification of anti-affine groups over algebraically
closed fields is obtained. Subsequently, he extended this
classification to arbitrary fields in \cite{SS08}, jointly with 
Fernando Sancho de Salas. The approach of \cite{Sa01,SS08} differs
from the present one, their key ingredient being the classification of
certain torsors over anti-affine varieties. The terminology is also
different: the variedades cuasi-abelianas of \cite{Sa01}, or
quasi-abelian varieties of \cite{SS08}, are called anti-affine groups
here. I warmly thank Carlos Sancho de Salas for making me aware of his
work, and for many interesting exchanges.

\bigskip

\noindent
{\bf Notation and conventions.}
Throughout this article, we denote by $k$ a field with separable
closure $k_s$ and algebraic closure $\kb$. The Galois group of $k_s$
over $k$ is denoted by $\Gamma_k$. A $\Gamma_k$-\emph{lattice} is a
free abelian group of finite rank equipped with an action of
$\Gamma_k$. 

By a \emph{scheme}, we mean a scheme of finite type over $k$, unless 
otherwise specified; a point of a scheme will always mean a closed
point. Morphisms of schemes are understood to be $k$-morphisms, and
products are taken over $k$. A \emph{variety} is a separated, 
geometrically integral scheme. 

We use \cite{SGA3} as a general reference for \emph{group schemes}.
However, according to our conventions, any group scheme $G$ is assumed 
to be of finite type over $k$. The group law of $G$ is denoted
multiplicatively, and $e_G$ stands for the neutral element of $G(k)$,
except for commutative groups where we use an additive notation.
By an \emph{algebraic group}, we mean a smooth group scheme $G$,
possibly non-connected. An \emph{abelian variety} is a connected and
complete algebraic group. For these, we refer to \cite{Mu70}, and to
\cite{Bo91} for affine algebraic groups. 

Given a connected group scheme $G$, we denote by $G_{\aff}$ 
the smallest normal connected affine subgroup scheme of $G$ such 
that the quotient $G/G_{\aff}$ is an abelian variety, and by
\begin{equation}\label{eqn:alb}
\alpha_G : G \to G/G_{\aff} =: A(G)
\end{equation}
the quotient homomorphism. The existence of $G_{\aff}$ is due 
to Chevalley in the setting of algebraic groups over algebraically 
closed fields; then $G_{\aff}$ is an algebraic group as well, 
see \cite{Ro56,Ch60}. Chevalley's result implies the 
existence of $G_{\aff}$ for any connected group scheme $G$, see 
\cite[Lem.~IX.2.7]{Ra70} or \cite[Thm.~9.2.1]{BLR90}.

Also, we denote by
\begin{equation}\label{eqn:aff}
\varphi_G : G \to \Spec \cO(G)
\end{equation}
the canonical morphism, known as the affinization of $G$. Then 
$\varphi_G$ is the quotient homomorphism by $G_{\ant}$, the largest 
anti-affine subgroup scheme of $G$. Moreover, $G_{\ant}$ is a 
connected algebraic subgroup of the centre of $G$ (see 
\cite[Sec.~III.3.8]{DG70} for these results).

\section{Basic properties}

\subsection{Characterizations of anti-affine groups}
\label{subsec:cha}

Recall that a group scheme $G$ over $k$ is affine if and only if 
$G$ admits a faithful linear representation in a finite-dimensional 
vector space; this is also equivalent to the affineness of 
the $K$-group scheme $G_K := G \otimes_k K$ for some field extension 
$K/k$. We now obtain analogous criteria for anti-affineness:

\begin{lemma}\label{lem:equ}
The following conditions are equivalent for a $k$-group scheme $G$:

\smallskip

\noindent
{\rm (i)} $G$ is anti-affine.

\smallskip

\noindent
{\rm (ii)} $G_K$ is anti-affine for some field extension $K/k$.

\smallskip

\noindent
{\rm (iii)} Every finite-dimensional linear representation of $G$ is
trivial.

\smallskip

\noindent
{\rm (iv)} Every action of $G$ on a variety $X$ containing a
fixed point is trivial.
\end{lemma}

\begin{proof}
(i)$\Leftrightarrow$(ii) follows from the isomorphism
$\cO(G_K) \simeq \cO(G) \otimes_k K$.

(i)$\Leftrightarrow$(iii) follows from the fact that every linear 
representation of $G$ factors through the affine quotient group scheme
$G/G_{\ant}$.

Since (iv)$\Rightarrow$(iii) is obvious, it remains to show 
(iii)$\Rightarrow$(iv). Let $x$ be a $G$-fixed point in $X$ with local
ring $\cO_x$ and maximal ideal $\fm_x$. Then each quotient
$\cO_x/\fm_x^n$ is a finite-dimensional $k$-vector space on which $G$
acts linearly, and hence trivially. Since $\bigcap_n \fm_x^n = \{ 0 \}$, 
it follows that $G$ fixes $\cO_x$ pointwise. Thus, $G$ acts trivially
on $X$.
\end{proof}

\begin{remark}\label{rem:aff}
The preceding argument yields another criterion for affineness 
of a group scheme; namely, the existence of a faithful action 
on a variety having a fixed point. This was first observed by
Matsumura (see \cite{Ma63}).
\end{remark}

Next, we show that the class of anti-affine groups is stable 
under products, extensions and quotients:

\begin{lemma}\label{lem:ext}
Let $G_1$, $G_2$ be connected group schemes. Then:

\smallskip

\noindent
{\rm (i)} $G_1 \times G_2$ is anti-affine if and only if $G_1$ 
and $G_2$ are both anti-affine.

\smallskip

\noindent
{\rm (ii)} Given an exact sequence of group schemes
$$
1 \longrightarrow G_1 \longrightarrow G \longrightarrow G_2
\longrightarrow 1,
$$
if $G$ is anti-affine, then so is $G_2$. Conversely, if 
$G_1$ and $G_2$ are both anti-affine, then so is $G$.
\end{lemma}

\begin{proof}
(i) follows from the isomorphism 
$\cO(G_1 \times G_2) \simeq \cO(G_1) \otimes_k \cO(G_2)$.

(ii) The isomorphism $\cO(G_2) \simeq \cO(G)^{G_1}$
(the algebra of invariants under the action of $G_1$ on $\cO(G)$ 
via left multiplication) yields the first assertion.

If $G_1$ is anti-affine, then its action on $\cO(G)$ is trivial
(as $\cO(G)$ is a union of finite-dimensional $k$-$G_1$-submodules, 
and $G_1$ acts trivially on any such module by Lemma \ref{lem:equ}). 
Thus, $\cO(G_2) \simeq \cO(G)$ which implies the second assertion.
\end{proof}

Anti-affineness is also stable under isogenies:

\begin{lemma}\label{lem:iso}
Let $f : G \to H$ be an isogeny of connected commutative algebraic 
groups. Then $G$ is anti-affine if and only if so is $H$.
\end{lemma}

\begin{proof}
If $G$ is anti-affine, then so is $H$ by Lemma \ref{lem:ext} (ii).
For the converse, note that $f$ induces an isogeny 
$G_{\ant} \to I$, where $I$ is a subgroup scheme of $H$,
and in turn an isogeny $G/G_{\ant} \to H/I$.
As $G/G_{\ant}$ is affine, so is $H/I$. But $H/I$
is also anti-affine, and hence is trivial. Thus, $f$ restricts to 
an isogeny $G_{\ant} \to H$. In particular, 
$\dim(G_{\ant}) = \dim(H) = \dim(G)$, whence $G_{\ant} = G$. 
\end{proof}

\subsection{Rigidity}
\label{subsec:rig}

In this subsection, we generalize some classical properties of 
abelian varieties to the setting of anti-affine groups. Our results
are implicit in \cite{Ro56,Se58a}; we give full proofs for the sake of
completeness. 

\begin{lemma}\label{lem:rig}
Let $G$ be an anti-affine algebraic group, and $H$ a connected group 
scheme.

\smallskip

\noindent
{\rm (i)} Any morphism (of schemes) $f : G \to H$
such that $f(e_G) = e_H$ is a homomorphism (of group schemes), 
and factors through $H_{\ant}$; in particular, through the centre of 
$H$.

\smallskip

\noindent
{\rm (ii)} The abelian group (for pointwise multiplication) of 
homomorphisms $f : G \to H$ is free of finite rank.
\end{lemma}

\begin{proof}
(i) Consider the quotient homomorphism (\ref{eqn:alb})
$$
\alpha_H : H \to H/H_{\aff} =: A(H).
$$ 
By rigidity of abelian varieties (see e.g. \cite[Lem.~2.2]{Co02}), 
the composition $\alpha_H f : G \to A(H)$ is a homomorphism. 
Equivalently, the morphism
$$
F : G \times G \longrightarrow H, \quad
(x,y) \longmapsto f(xy) f(x)^{-1} f(y)^{-1}
$$
factors through the affine subgroup scheme $H_{\aff}$. 
As $G \times G$ is anti-affine, and $F(e_G,e_G) = e_H$, it follows 
that $F$ factors through $e_H$; thus, $f$ is a homomorphism. 

The composition of $f$ with the homomorphism (\ref{eqn:aff})
$$
\varphi_H : H \to H/H_{\ant}
$$
is a homomorphism from $G$ to an affine group scheme. Hence 
$\varphi_H f$ factors through $e_{H/H_{\ant}}$, that is, 
$f$ factors through $H_{\ant}$.

(ii) We may assume that $k$ is algebraically closed; then
$G_{\aff}$ is a connected affine algebraic group. By
\cite[Lem.~2.3]{Co02}, it follows that any homomorphism 
$f : G \to H$ fits into a commutative square
$$
\CD
G @>{f}>> H \\
@V{\alpha_G}VV @V{\alpha_H}VV \\
A(G) @>{\alpha(f)}>> A(H) \\
\endCD
$$
where $\alpha(f)$ is a homomorphism. This yields a homomorphism
$$
\alpha : \Hom(G,H) \longrightarrow \Hom(A(G),A(H)), 
\quad f \longmapsto \alpha(f).
$$
If $\alpha(f) = 0$, then $f$ factors through $H_{\aff}$, and hence
is trivial. Thus, $\Hom(G,H)$ is identified to a subgroup of
$\Hom(A(G),A(H))$; the latter is free of finite rank by
\cite[p.~176]{Mu70}.
\end{proof}

Next, we show that anti-affine groups are ``divisible'' (this property 
is the main ingredient of the classification of anti-affine groups in 
positive characteristics):

\begin{lemma}\label{lem:end}
Let $G$ be an anti-affine algebraic group, and $n$ a non-zero integer. 
Then the multiplication map $n_G : G \to G$, $x \mapsto nx$ is an isogeny.
\end{lemma}

\begin{proof}
Let $H$ denote the cokernel of $n_G$; then $n_H$ is trivial. 
Hence the abelian variety $H/H_{\aff}$ is trivial, i.e., 
$H$ is affine. But $H$ is anti-affine as a quotient of $G$, so that 
$H$ is trivial.
\end{proof}

\section{Structure}

\subsection{Semi-abelian varieties}
\label{subsec:sab}

Throughout this section, we consider connected group schemes $G$ 
equipped with an isomorphism
$$
\CD
\alpha : G/G_{\aff} @>{\simeq}>> A \\
\endCD
$$
where $A$ is a prescribed abelian variety.
We then say that $G$ is a \emph{group scheme over $A$}. 

Our aim is to classify anti-affine groups over $A$, up to 
isomorphism of group schemes over $A$ (in an obvious sense). We 
begin with the case where $G_{\aff}$ is a torus, i.e., $G$ is a 
semi-abelian variety. Then $G$ is obtained as an extension of
algebraic groups
\begin{equation}\label{eqn:sab}
\CD
1 @>>> T @>>> G @>{\alpha}>> A @>>> 0
\endCD
\end{equation}
where $T$ is a torus. Moreover, as $T_{k_s} := T \otimes_k k_s$ is
split, we have a decomposition of quasi-coherent sheaves on $A_{k_s}$: 

\begin{equation}\label{eqn:dec}
\alpha_*(\cO_{G_{k_s}}) = 
\bigoplus_{\lambda \in \Lambda} \cL_{\lambda}
\end{equation}
where $\Lambda$ denotes the character group of $T$
(so that $\Lambda$ is a $\Gamma_k$-lattice), and 
$\cL_{\lambda}$ consists of all eigenvectors of $T_{k_s}$ in 
$\alpha_*(\cO_{G_{k_s}})$ with weight $\lambda$. Each 
$\cL_{\lambda}$ is an invertible sheaf on $A_{k_s}$, 
algebraically equivalent to $0$. Thus, 
$\cL_{\lambda}$ yields a $k_s$-rational point $c(\lambda)$ 
of the dual abelian variety $A^{\vee}$. Moreover, the map
\begin{equation}\label{eqn:c}
c : \Lambda \to A^{\vee}(k_s), 
\quad \lambda \mapsto c(\lambda)
\end{equation}
is a $\Gamma_k$-equivariant homomorphism, which classifies the 
extension (\ref{eqn:sab}) up to isomorphism of extensions 
(as follows e.g. from \cite[VII.3.16]{Se59}). In other words,
the extensions (\ref{eqn:sab}) are classified by the homomorphisms 
$T^{\vee} \to A^{\vee}$, where $T^{\vee}$ denotes the Cartier dual
of $T$. 

\begin{proposition}\label{prop:sab}
{\rm (i)} With the preceding notation, $G$ is anti-affine if 
and only if $c$ is injective.

\smallskip

\noindent
{\rm (ii)} The isomorphism classes of anti-affine semi-abelian 
varieties over $A$ correspond bijectively to the 
sub-$\Gamma_k$-lattices of $A^{\vee}(k_s)$.
\end{proposition}

\begin{proof}
(i) By the decomposition (\ref{eqn:dec}), we have
$$
\cO(G_{k_s}) = H^0 \big( A_{k_s},\alpha_*(\cO_{G_{k_s}}) \big) 
= \bigoplus_{\lambda \in \Lambda} H^0(A_{k_s},\cL_{\lambda})
$$
and of course $H^0(A_{k_s}, \cL_0) = \cO(A_{k_s}) = k_s$. 
Thus, $G$ is anti-affine if and only if 
$H^0(A_{k_s},\cL_{\lambda}) = 0$ for all $\lambda \neq 0$.

On the other hand, $H^0(A_{k_s},\cL) = 0$ for any invertible 
sheaf $\cL$ on $A_{k_s}$ which is algebraically trivial but 
non-trivial. (Otherwise, $\cL_{\kb} = \cO_{A_{\kb}}(D)$ for some 
non-zero effective divisor $D$ on $A_{\kb}$. We may find 
an integral curve $C$ in $A_{\kb}$ that meets properly
$\Supp(D)$. Then the pull-back of $\cL$ to $C$ has positive degree, 
contradicting the algebraic triviality of $\cL$).

Thus, $G$ is anti-affine if and only if $\cL_{\lambda}$ is 
non-trivial for any $\lambda \ne 0$.

(ii) Given two injective and $\Gamma_k$-equivariant homomorphisms
$$
c_1,c_2 : \Lambda \longrightarrow A^{\vee}(k_s),
$$ 
the corresponding anti-affine groups are isomorphic over $A$ 
if and only if the corresponding extensions differ by an 
automorphism of $T$, i.e., there exists a $\Gamma_k$-equivariant 
automorphism $f$ of $\Lambda$ such that $c_2 = c_1 f$. 
This amounts to the equality $c_1(\Lambda) = c_2(\Lambda)$.
\end{proof}

In positive characteristics, the preceding construction yields 
all anti-affine groups:

\begin{proposition}\label{prop:pos}
Any anti-affine algebraic group over a field of characteristic 
$p > 0$ (resp.~over a finite field) is a semi-abelian variety
(resp.~an abelian variety). 
\end{proposition}

\begin{proof}
The multiplication map $p_G$ is an isogeny by Lemma
\ref{lem:end}. In particular, the group $G(\kb)$ contains only 
finitely many points of order $p$. Thus, every unipotent 
subgroup of $G_{\kb}$ is trivial.  By 
\cite[Exp.~XVII, Thm.~7.2.1]{SGA3}, it follows that 
$(G_{\kb})_{\aff}$ is a torus, i.e., $G_{\kb}$ is a semi-abelian 
variety. Hence $G$ is a semi-abelian variety as well, see 
\cite[p.~178]{BLR90}.   

If $k$ is finite (so that $k_s = \kb$, then the group $A^{\vee}(k_s)$
is the union of its subgroups $A^{\vee}(K)$, where $K$ ranges over all 
finite subfields of $k_s$ that contain $k$. As a consequence, every
point of $A^{\vee}(k_s)$ has finite order. Hence any sublattice of
$A^{\vee}(k_s)$ is trivial.
\end{proof}

\subsection{Vector extensions of abelian varieties}
\label{subsec:vec}

In this subsection, we assume that $k$ has characteristic $0$.
Recall that every abelian variety $A$ has a 
\emph{universal vector extension} $E(A)$ by the $k$-vector space
$H^1(A,\cO_A)^*$ regarded as an additive group. In other words, 
any extension $G$ of $A$ by a vector group $U$ fits into a unique 
commutative diagram
\begin{equation}\label{eqn:cd}
\CD
0 @>>> H^1(A,\cO_A)^* @>>> E(A) @>>>     A        @>>> 0 \\
&   &  @V{\gamma}VV        @VVV @V{\id}VV \\
0 @>>> U              @>>> G   @>>>      A        @>>> 0 \\
\endCD
\end{equation}
(see \cite{Ro58, Se59, MM74}). Note that 
$E(A)_{\aff} = H^1(A,\cO_A)^*$. 

\begin{proposition}\label{prop:vec}
{\rm (i)} $E(A)$ is anti-affine.

\smallskip

\noindent
{\rm (ii)} With the notation of the diagram (\ref{eqn:cd}), 
$G$ is anti-affine if and only if the classifying map 
$\gamma : H^1(A,\cO_A)^* \to U$ is surjective. 

\smallskip

\noindent
{\rm (iii)} The anti-affine groups over $A$ obtained as vector 
extensions are classified by the subspaces of the $k$-vector space
$H^1(A,\cO_A)$.
\end{proposition}

\begin{proof}
(i) The affinization epimorphism (\ref{eqn:aff})
$$
\varphi  = \varphi_{E(A)} : E(A) \to V
$$ 
induces epimorphisms 
$$
H^1(A,\cO_A)^* \to W, \quad 
A = E(A)/H^1(A,\cO_A)^* \to V/W
$$
where $W$ is a subspace of $V$. The latter epimorphism must be
trivial, and hence $\varphi$ restricts to an epimorphism
$$
\delta : H^1(A,\cO_A)^* \to V.
$$
Moreover, $V$ is a vector group, and $\delta$ is $k$-linear.
The extension given by the commutative diagram
$$
\CD
0 @>>> H^1(A,\cO_A)^* @>>> E(A) @>>>  A @>>> 0 \\
&  &   @V{\delta}VV        @VVV       @V{\id}VV \\
0 @>>> V              @>>> H    @>>>  A @>>> 0 \\
\endCD
$$
is split, as the map $- \varphi + \id : E(A) \times V \to V$ factors 
through a retraction of $H$ onto $V$. Since $E(A)$ is the universal 
extension, it follows that $\delta = 0$, i.e., $V = 0$.

(ii) The group $G$ is the quotient of $E(A) \times U$ by the diagonal 
image of $H^1(A,\cO_A)^*$. Since $\cO \big(E(A) \big) = k$, it follows 
that $\cO(G)$ is the algebra of invariants of 
$\cO(U)$ under $H^1(A,\cO_A)^*$ acting by translations
via $\gamma$. This implies the assertion.

(iii) follows from (ii) by assigning to $\gamma$ the image of the 
transpose map $\gamma^t : U^* \to H^1(A,\cO_A)$. 
\end{proof}

\begin{remark}
In the preceding statement, the assumption of characteristic $0$
cannot be omitted in view of Proposition \ref{prop:pos}. 
This may also be seen directly as follows.
If $k$ has characteristic $p > 0$, any vector extension
$0 \to U \to G \to A \to 0$ 
splits after pull-back under the multiplication map 
$p_A : A \to A$ (since $p_A$ is an isogeny, and $p_U = 0$).
This yields an isogeny $U \times A \to G$. Thus, $G$ cannot be
anti-affine in view of Lemma \ref{lem:iso}.
\end{remark}

\subsection{Classification of anti-affine groups}
\label{subsec:cla}

To complete this classification, we may assume that $k$ has 
characteristic $0$, in view of Proposition \ref{prop:pos}. 

Let $G$ be an anti-affine algebraic group. Then $G_{\aff}$ is a 
connected commutative algebraic group, and hence admits a unique 
decomposition
\begin{equation}\label{eqn:prod}
G_{\aff} = T \times U
\end{equation}
where $T$ is a torus, and $U$ is connected and unipotent; $U$ has
then a unique structure of $k$-vector space. Thus, $G/U$ is a 
semi-abelian variety (extension of $A$ by $T$) and $G/T$ is a 
vector extension of $A$ by $U$. Moreover, the quotient homomorphisms 
$p_U : G \to G/U$, $p_T : G \to G/T$ fit into a cartesian square
\begin{equation}\label{eqn:cart}
\CD
G  @>{p_U}>>  G/U            \\
@V{p_T}VV     @V{\alpha_{G/U}}VV    \\
G/T @>{\alpha_{G/T}}>> A       \\
\endCD
\end{equation}
where $\alpha_{G/U}$ (resp.~$\alpha_{G/T}$) is the quotient by $T$
(resp.~$U$). Moreover,
$\alpha = \alpha_{G/T} p_T = \alpha_{G/U} p_U$.
This yields a canonical isomorphism of algebraic groups
over $A$:
\begin{equation}\label{eqn:gtu}
\CD
G @>{\simeq}>> G/U \times_A G/T. \\
\endCD
\end{equation}

\begin{proposition}\label{prop:quot}
With the preceding notation, $G$ is anti-affine if and only 
if $G/U$ and $G/T$ are both anti-affine.
\end{proposition}

\begin{proof}
If $G$ is anti-affine, then so are its quotient groups $G/U$ and $G/T$.

For the converse, we may assume that $k$ is algebraically closed in
view of Lemma \ref{lem:equ}. Note that the diagram (\ref{eqn:cart})
yields an isomorphism of quasi-coherent sheaves on $A$:
\begin{equation}\label{eqn:tens}
\alpha_* (\cO_G) \simeq
\alpha_{G/U,*} (\cO_{G/U}) \otimes_{\cO_A} \alpha_{G/T,*} (\cO_{G/T}).
\end{equation}
Moreover, we have a decomposition
$$
\alpha_{G/U,*} (\cO_{G/U}) = 
\bigoplus_{\lambda \in \Lambda} \cL_{\lambda}
$$
as in (\ref{eqn:dec}), where $\cL_0 = \cO_A$ while
$H^0(A,\cL_{\lambda}) = 0$ for any $\lambda \neq 0$.
On the other hand, the quasi-coherent sheaf $\alpha_{G/T,*} (\cO_{G/T})$ 
admits an increasing filtration with subquotients isomorphic 
to $\cO_A$, by the next lemma applied to the $U$-torsor
$\alpha_{G/T} : G/T \to A$. It follows that
$$
H^0 \big( A, 
\cL_{\lambda} \otimes_{\cO_A} \alpha_{G/T,*} (\cO_{G/T}) \big) = 0
$$
for any $\lambda \neq 0$. Thus,
$$\displaylines{
\cO(G) = H^0 \big( A,\alpha_*(\cO_G) \big) \simeq
\bigoplus_{\lambda \in \Lambda} H^0 \big( A, 
\cL_{\lambda} \otimes_{\cO_A} \alpha_{G/T,*} (\cO_{G/T}) \big)
\hfill \cr \hfill 
= H^0 \big( A,\alpha_{G/T,*} (\cO_{G/T}) \big) = \cO(G/T) = k.
\cr}
$$
\end{proof}

\begin{lemma}\label{lem:fib}
Let $\pi : X \to Y$ be a torsor under a non-trivial connected 
unipotent algebraic group $U$. Then the quasi-coherent sheaf 
$\pi_* (\cO_X)$ admits an infinite increasing filtration with
subquotients isomorphic to~$\cO_Y$.
\end{lemma}

\begin{proof}
We claim that there is an isomorphism of quasi-coherent sheaves 
over $Y$:
$$
\CD
u : \pi_*(\cO_X) @>{\simeq}>> 
\pi_* \big( \cO_X \otimes_k \cO(U) \big)^U. \\
\endCD
$$
Here the right-hand side denotes the subsheaf of $U$-invariants
in the quasi-coherent sheaf $\pi_* \big( \cO_X \otimes_k \cO(U) \big)$, 
where $U$ acts via its natural action on $\cO_X$ and its action 
on $\cO(U)$ by left multiplication.

The assertion of the lemma follows from that claim, as the 
$U$-module $\cO(U)$ admits an infinite increasing filtration 
with trivial subquotients.

To prove the claim, we first construct a natural isomorphism
$$
\CD
u_M : M @>{\simeq}>> \big( M \otimes_k \cO(U) \big)^U \\
\endCD
$$
for any $U$-module $M$. Indeed, the right-hand side may be 
regarded as the space of $U$-equivariant morphisms $f : U \to M$. 
Any such morphism is of the form $f_m : u \to u \cdot m$ for a 
unique $m \in M$, namely, $m = f(e_U)$. We then set 
$u_M(m) := f_m$.

Next, if the $U$-module $M$ is also a $k$-algebra where $U$ acts 
by algebra automorphisms, then $u_M$ is an isomorphism of 
$M^U$-algebras, where the algebra of invariants $M^U$ acts on 
$\big( M \otimes_k \cO(U) \big)^U$ via multiplication on $M$.
Moreover, $u_M$ commutes with localization by elements of $M^U$.
Thus, the isomorphisms $u_{\cO(\pi^{-1}(Y_i))}$, where 
$(Y_i)_{i \in I}$ is an affine open covering of $Y$, may be glued 
to yield the desired isomorphism.
\end{proof}

Combining the results of Propositions \ref{prop:pos},
\ref{prop:vec} and \ref{prop:quot}, we obtain the following 
classification:

\begin{theorem}\label{thm:cla}
{\rm (i)} In positive characteristics, the isomorphism classes 
of anti-affine groups over an abelian variety $A$ correspond 
bijectively to the sub-$\Gamma_k$-lattices 
$\Lambda \subset A^{\vee}(k_s)$.

\smallskip

\noindent
{\rm (ii)} In characteristic $0$, the isomorphism classes of 
anti-affine groups over $A$ correspond bijectively to the pairs 
$(\Lambda,V)$, where $\Lambda$ is as in {\rm (i)}, and $V$ is a 
subspace of $H^1(A,\cO_A)$.
\end{theorem}

\begin{remark}
(i) The preceding classification may be formulated in terms of 
the dual variety $A^{\vee}$ only, as $H^1(A,\cO_A)$ is naturally
isomorphic to the tangent space $T_0(A^{\vee})$ (the Lie algebra
of $A^{\vee}$); see e.g. \cite[p.~130]{Mu70}. 

\smallskip

\noindent
(ii) To classify the anti-affine groups $G$ without prescribing 
an isomorphism $G/G_{\aff} \simeq A$, it suffices to replace 
the sublattices $\Lambda$ (resp.~the pairs $(\Lambda,V)$) with 
their isomorphism classes under the natural action 
of the automorphism group $\Aut(A)$ of the abelian variety $A$
(resp.~of the natural action of $\Aut(A) \times \Aut(A)$ on pairs). 
\end{remark}

\subsection{Universal morphisms}
\label{subsec:uni}

Throughout this subsection, we assume that the ground field $k$ is
perfect. We investigate morphisms from a prescribed variety
to anti-affine algebraic groups, by adapting results and
methods of Serre (see \cite{Se58a,Se58b}).

Consider a variety $X$ equipped with a $k$-rational point $x$. Then 
there exists a universal morphism to a semi-abelian variety
$$
\sigma_{X,x} : X \longrightarrow S, \quad x \longmapsto e_S.
$$
Indeed, this is a special case of \cite[Thm.~7]{Se58a} when $k$ is
algebraically closed, and the case of a perfect field follows
by Galois descent as in \cite[Sec.~V.4]{Se59} (see 
\cite[Thm.~A.1]{Wit06} for a generalization to an arbitrary field). 

We say that $\sigma_{X,x}$ is the \emph{generalized Albanese morphism}
of the pointed variety $(X,x)$, and $S = S_X$ is the 
\emph{generalized Albanese variety}, which indeed depends only on
$X$. The formation of $\sigma_{X,x}$ commutes with base change to
perfect field extensions.

Recalling the extension of algebraic groups (\ref{eqn:sab})
$$
\CD 
1 @>>> T @>>> S @>{\alpha_S}>> A @>>> 0, 
\endCD
$$ 
the composite morphism
$$
\alpha_{X,x} := \alpha_S \sigma_{X,x} : X \longrightarrow A = A_X
$$
is the \emph{Albanese morphism} of $X$, i.e., the universal morphism 
to an abelian variety that maps $x$ to the origin.

We note that the pull-back map
$$
\alpha_{X,x}^* : A^{\vee}(k) \subset \Pic(A) \longrightarrow \Pic(X)
$$
is independent of the choice of $x \in X(k)$ (indeed, any two Albanese 
morphisms differ by a translation by a $k$-rational point of $A$, 
and the translation action of $A$ on $A^{\vee}$ is trivial); we denote
that map by $\alpha_X^*$. Likewise, the analogous map 
$H^1(A,\cO_A) \to H^1(X,\cO_X)$ is independent of $x$.

We also record the following observation:

\begin{lemma}\label{lem:com}
Let $X$ be a complete variety equipped with a $k$-rational point. Then
the pull-back maps $A^{\vee}(k) \to \Pic(X)$ and  
$H^1(A,\cO_A) \to H^1(X,\cO_X)$ are both injective.
\end{lemma}

\begin{proof}
This follows from general results on the Picard functor
(see \cite[Chap.~8]{BLR90}); we provide a direct argument. 
We may assume that $k$ is algebraically closed. Let 
$L \in A^{\vee}(k)$ such that $\alpha_X^*(L) = 0$. 
Consider the corresponding extension
\begin{equation}\label{eqn:tors}
\CD
1 @>>> \bG_m @>>> G @>{\alpha}>> A  @>>>0
\endCD
\end{equation}
as a $\bG_m$-torsor over $A$. Then the pull-back of this
torsor under $\alpha_{X,x}$ is trivial, that is, the projection
$X \times_A G \to X$ has a section. Thus, $\alpha_{X,x}$ lifts to
a morphism $\gamma : X \to G$ and hence to a morphism 
$X \to H$ where $H \subset G$ denotes the algebraic subgroup 
generated by the image of $\gamma$. Since $X$ is complete, $H$ is an
abelian variety (as follows e.g. from \cite[Exp.~VIB, Prop.~7.1]{SGA3}). 
Thus, $\alpha$ restricts to an isogeny $\beta : H \to A$. 
By the universal property of the Albanese morphism, it follows that 
$\beta$ is an isomorphism. Thus, the extension (\ref{eqn:tors}) is
split; in 
other words, $L$ is trivial. 

Likewise, given $u \in H^1(A,\cO_A)$ such that $\alpha_X^*(u) = 0$, 
one checks that $u = 0$ by considering the associated extension
$$
0 \longrightarrow \bG_a \longrightarrow G \longrightarrow A
\longrightarrow 0.
$$
\end{proof}

We now obtain a criterion for anti-affineness of the generalized
Albanese variety:

\begin{proposition}\label{prop:uni}
Given a pointed variety $(X,x)$, the associated semi-abelian variety 
$S = S_X$ is anti-affine if and only if $\cO(X_{\kb})^* = \kb^*$. 

Under that assumption, $S$ is classified by the pair $(A,\Lambda)$,
where $A$ is the Albanese variety of $X$, and $\Lambda$ is the kernel
of the pull-back map
$$
\alpha_X^* : A^{\vee}(\kb) \longrightarrow \Pic(X_{\kb}).
$$
In particular, this kernel is a $\Gamma_k$-lattice.
\end{proposition}

\begin{proof}
Denote by $\varphi : S \to S/S_{\ant}$ the affinization morphism
(\ref{eqn:aff}). Then $S/S_{\ant}$ is affine and semi-abelian, hence a
torus. Clearly, the composite $\varphi \sigma_{X,x}$ is the
universal morphism from $X$ to a torus, that maps $x$ to the neutral
element. 

Given a $\Gamma_k$-lattice $M$, the morphisms from $X$ to the dual
torus $M^{\vee}$ correspond bijectively to the $\Gamma_k$-equivariant
homomorphisms $M \to \cO(X_{\kb})^*$. Moreover, the exact sequence of 
$\Gamma_k$-modules
$$
1 \longrightarrow \kb^* \longrightarrow \cO(X_{\kb})^* \longrightarrow  
\cO(X_{\kb})^*/\kb^* \longrightarrow 1
$$
is split by the evaluation map at $x \in X(k)$, and 
$\cO(X_{\kb})^*/\kb^*$ is a $\Gamma_k$-lattice. Thus, the morphisms of
pointed varieties 
$$
(X,x) \longrightarrow (M^{\vee},e_{M^{\vee}})
$$ 
correspond bijectively to the $\Gamma_k$-equivariant homomorphisms
$$
M \longrightarrow \cO(X_{\kb})^*/\kb^*.
$$  
In particular, there is a universal such morphism, to the dual torus
of the lattice $\cO(X_{\kb})^*/\kb^*$. This yields the first assertion.

Assuming that $\cO(X_{\kb})^* = \kb^*$, consider a 
sub-$\Gamma_k$-lattice $M \subset A^{\vee}(\kb)$ and the 
corresponding extension 
$$
\CD
1 @>>> M^{\vee} @>>> G @>>> A @>>> 0.
\endCD
$$ 
We regard $G$ as a $M^{\vee}$-torsor over $A$. Then the 
morphisms $\gamma : X \to G$ that lift the Albanese morphism
$\alpha_{X,x} : X \to A$ are identified to the sections of
the pull-back $M^{\vee}$-torsor $X \times_A S \to X$, as in the proof
of Lemma \ref{lem:com}. Such sections exist if and only if the
pull-back map $M \to \Pic(X_{\kb})$ is trivial; moreover,
any two sections differ by a morphism $X \to T$, i.e. 
by a $k$-rational point of $T$. Thus, there exists a unique section
such that the associated morphism $\gamma$ maps $x$ to $e_S$.

As a consequence, the liftings of $\alpha_{X,x}$ to semi-abelian 
varieties over $A$ are classified by the sub-$\Gamma_k$-lattices 
of $\Lambda := \ker(\alpha_X^*)$. We now show that $\Lambda$
is a $\Gamma_k$-lattice, thereby completing the proof. 
For this, we may assume that $k$ is algebraically closed.

If $X$ is complete, then $\Lambda$ is trivial by Lemma
\ref{lem:com}. In the general case, let $i : X \to \overline{X}$ be an
open immersion into a complete variety. We may assume that
$\alpha_{X,x}$ extends to a morphism 
$\alpha_{\overline{X},x} : \overline{X} \to A$; then
$\alpha_{\overline{X},x}$ is the Albanese morphism of
$(\overline{X},i(x))$. Since $\alpha_{\overline{X}}^*$ is injective,
$\Lambda$ is identified to a subgroup of the kernel of
$$
i^* : \Pic(\overline{X}) \to \Pic(X).
$$
But $\ker(i^*)$ is the group of Cartier divisors with support in 
$\overline{X} \setminus X$, as $\cO(X)^* = k^*$. In particular, the
abelian group $\ker(i^*)$ is free of finite rank, and hence so is
$\Lambda$.
\end{proof}

By another result of Serre (see \cite[Thm.~8]{Se58a}), 
a pointed variety $(X,x)$ admits a universal morphism to 
a commutative algebraic group,
$$
\gamma_{X,x} : X \longrightarrow G, \quad x \longmapsto e_G
$$
if and only if $\cO(X) = k$, that is, $X$ is anti-affine. 
Then $G$ is also anti-affine, as this group is generated over 
$\kb$ by the image of $X$. In positive characteristics, the universal
group $G$ is just the generalized Albanese variety, by Proposition
\ref{prop:pos}. In characteristic $0$, this group may be described as 
follows: 

\begin{proposition}\label{prop:dat}
Let $(X,x)$ be a pointed anti-affine variety over a field $k$ of
characteristic $0$ and let $G$ be the associated anti-affine
group. Then $G$ is classified by the triple $(A,\Lambda,V)$ where $A$
and $\Lambda$ are as in the preceding proposition, and $V$ is the
kernel of the pull-back map
$\alpha_X^* : H^1(A,\cO_A) \longrightarrow H^1(X,\cO_X)$.
\end{proposition}

The proof is analogous to that of Proposition \ref{prop:uni},
taking into account the isomorphism (\ref{eqn:gtu}) and the 
structure of anti-affine vector extensions of $A$.

\begin{remarks}
(i) The associated data $A,\Lambda,V$ may be described explicitly 
in terms of completions, for smooth varieties in characteristic $0$.
Namely, given such a variety $X$, there exists an open immersion 
$i : X \to \overline{X}$ where $\overline{X}$ is a smooth complete
variety. Then $\Pic^0(\overline{X})$ is an abelian variety with dual
the Albanese variety $A_X = A_{\overline{X}}$.

If $\cO(X_{\kb})^*  = \kb^*$, then the $\Gamma_k$-lattice 
$\Lambda$ of Proposition \ref{prop:uni} is the group of divisors 
supported in $\overline{X} \setminus X$ and algebraically equivalent
to $0$, by the arguments in \cite[Sec.~1]{Se58b}. 

Under the (stronger) assumption that $\cO(X) = k$, the subspace
$V \subset H^1(A,\cO_A)$ of Proposition \ref{prop:dat} equals
$$
H^1_{\overline{X} \setminus X}(\overline{X}, \cO_{\overline{X}}) 
= H^0(\overline{X}, i_*(\cO_X)/\cO_{\overline{X}}),
$$
as follows from similar arguments.

\smallskip

\noindent
(ii) Dually, one may also consider morphisms from varieties, or
schemes, to a prescribed anti-affine group $G$. In fact, such a group
admits a modular interpretation, which generalizes the duality of
abelian varieties.

To state it, recall that any abelian variety $A$ classifies the
invertible sheaves on $A^{\vee}$, algebraically equivalent to $0$ and
equipped with a rigidification along the zero section.

The universal extension $E(A)$ has also a modular interpretation: it
classifies the algebraically trivial invertible sheaves on $A^{\vee}$,
equipped with a rigidification along the first infinitesimal
neighbourhood $T_0(A^{\vee})$ (see \cite[Prop.~2.6.7]{MM74}).    

It follows that the algebraically trivial invertible sheaves
on $A^{\vee}$, equipped with rigidifications along a basis of the lattice
$\Lambda$ and along the subspace $V \subset T_0(A^{\vee})$, are classified 
by an anti-affine algebraic group over $A$ with data $(\Lambda,V)$.
\end{remarks}

\section{Some consequences}

\subsection{The Rosenlicht decomposition}
\label{subsec:ros}

We first obtain a variant of a structure theorem for algebraic groups
due to Rosenlicht (see \cite[Cor.~5, p.~440]{Ro56}), in the setting  
of group schemes.

\begin{proposition}\label{prop:ros}
Let $G$ be a connected group scheme over a field $k$. Then:

\smallskip

\noindent 
{\rm (i)} The group law of $G$ yields an exact sequence of group 
schemes
\begin{equation}\label{eqn:ros}
1 \longrightarrow G_{\aff} \cap G_{\ant} \longrightarrow
G_{\aff} \times G_{\ant} \longrightarrow G \longrightarrow 1.
\end{equation}
In other words, we have the decomposition $G = G_{\aff} \, G_{\ant}$. 

\smallskip

\noindent
{\rm (ii)} The connected subgroup scheme 
$(G_{\ant})_{\aff} \subset G_{\ant}$ 
is an algebraic group, contained in $G_{\aff}$; moreover, 
the quotient $(G_{\aff} \cap G_{\ant})/(G_{\ant})_{\aff}$ is finite.

\smallskip

\noindent
{\rm (iii)} The quotient group scheme $G' := G/(G_{\ant})_{\aff}$ 
has the decomposition $G' = G'_{\ab} \, G'_{\aff}$ where 
$G'_{\ab} = G_{\ant}/(G_{\ant})_{\aff}$ is the largest abelian 
subvariety of $G'$, and $G'_{\aff} = G_{\aff}/(G_{\ant})_{\aff}$.

\smallskip

\noindent
{\rm (iv)} Any subgroup scheme $H \subset G$ such that 
$G = G_{\aff} \, H$ contains $G_{\ant}$. 
\end{proposition}

\begin{proof}
(i) Since $G_{\aff}$ is a normal subgroup scheme of $G$, and 
$G_{\ant}$ is contained in the centre of $G$, we see that 
the multiplication map $G_{\aff} \times G_{\ant} \to G$ is a
homomorphism with kernel isomorphic to $G_{\aff} \cap G_{\ant}$; 
the image $G_{\aff} \, G_{\ant}$ is a normal subgroup scheme of $G$ 
by \cite[Exp.~VIA 5.3, 5.4]{SGA3}. 
The quotient $G /(G_{\aff} \, G_{\ant})$ is affine, as 
a quotient of $G/G_{\ant}$; but it is also an abelian variety, as 
a quotient of $G/G_{\aff}$. Thus, this quotient is trivial.

(ii) The smoothness of $(G_{\ant})_{\aff}$ follows from 
Proposition \ref{prop:pos}. By rigidity (see e.g. 
\cite[Lem.~2.2]{Co02}), every homomorphism from $(G_{\ant})_{\aff}$ 
to an abelian variety is trivial. As a consequence, 
$(G_{\ant})_{\aff} \subset G_{\aff}$.

The scheme $(G_{\aff} \cap G_{\ant})/(G_{\ant})_{\aff}$ is 
affine (as a quotient of a subgroup scheme of $G_{\aff}$) 
and proper (as a subgroup scheme of the abelian variety
$G_{\ant}/(G_{\ant})_{\aff}$). Hence this scheme is finite.

(iii) follows readily from (i) and (ii).

(iv) Note that $G = G_{\aff} \, H^0$, as $G$ is connected. Thus, 
$$
G = G_{\aff} \, H^0_{\aff}\,  H^0_{\ant}
$$
and $H^0_{\ant} \subset G_{\ant}$; in particular, $H^0_{\ant}$
is contained in the centre of $G$. On the other hand, 
$G_{\aff} \, H^0_{\aff}$ is affine, so that
$$
G/H^0_{\ant} \simeq (G_{\aff} \, H^0_{\aff}) / 
\big( H^0_{\ant} \cap(G_{\aff} \, H^0_{\aff}) \big) 
$$
is affine as well. Since the quotient homomorphism 
$G \to G/G_{\ant}$ is the affinization, it follows that 
$H^0_{\ant}$ contains $G_{\ant}$.
\end{proof}

Next, we consider the functorial properties of the Rosenlicht 
decomposition. By the results of \cite[III.3.8]{DG70}, 
the formation of $G_{\ant}$ commutes with base change to 
arbitrary field extensions, and with homomorphisms of group 
schemes. Also, note that the homomorphism (\ref{eqn:aff}) 
$\varphi_G : G \to \Spec \cO(G) = G/G_{\ant}$ depends only on $G$
regarded as a scheme. In particular, $G_{\ant}$ depends only on the 
pointed scheme $(G,e_G)$.

These properties are also satisfied by $G_{\aff}$ under additional 
assumptions. Specifically, if $G$ is a connected algebraic group
over a perfect field $k$, then $G_{\aff}$ is the largest connected
affine algebraic subgroup of $G$; the formation of $G_{\aff}$ commutes 
with base change to any perfect field extension of $k$ and with 
homomorphisms of algebraic groups (see \cite{Co02} for these 
results). The quotient homomorphism $\alpha_G : G \to G/G_{\aff}$ 
is the Albanese morphism of the pair $(G,e_G)$. In particular,
$G_{\aff}$ depends only on the pointed variety $(G,e_G)$.   

The assumption that $k$ is perfect cannot be omitted in view of the 
following example, obtained by a construction of Raynaud (see
\cite[Exp.~XVII, App.~III, Prop.~5.1]{SGA3}):

\begin{example}\label{ex:npe}
Let $k$ be a non-perfect field of characteristic $p > 0$
and choose a finite, non-trivial field extension $K/k$ such that 
$K^p \subset k$. Given a non-trivial abelian variety $A$ over 
$k$, let $A_K := A \otimes_k K$ (a non-trivial abelian variety 
over $K$) and 
$$
G := \Pi_{K/k}(A_K)
$$
where $\Pi_{K/k}$ denotes the Weil restriction; in other words, 
$G$ is the unique $k$-scheme such that 
\begin{equation}\label{eqn:uni}
G(R) = A_K(R \otimes_k K)
\end{equation} 
for any $k$-algebra $R$. Then $G$ is a commutative connected 
algebraic $k$-group, as follows e.g. from the results of 
\cite[A.2]{Oe84} that we shall use freely.

We claim that 
\begin{equation}\label{eqn:prK}
G_K = U \times A_K
\end{equation} 
where $U$ is a connected unipotent algebraic $K$-group; 
in particular,  $(G_K)_{\aff} = U$ and $(G_K)_{\ant} = A_K$. Moreover, 
$G_{\ant} = A$ but $G_{\aff}$ is not smooth, and 
$(G_{\aff})_K \neq (G_K)_{\aff}$.

Indeed, for any $K$-algebra $R$, we have
$$
G_K(R) = G(R) = A_K \big( R \otimes_K (K \otimes_k K) \big)
$$
and $K \otimes_k K$ is a finite-dimensional $K$-algebra.
The multiplication map $\mu : K \otimes_k K \to K$ yields 
an exact sequence
$$
0 \to \fm \to K \otimes_k K  \to K \to 0
$$
and the ideal $\fm$ is nilpotent, as 
$(x \otimes 1 - 1 \otimes x)^p = 0$ for any $x \in K$.
This yields a functorial morphism  $G_K(R) \to A_K(R)$ and, 
in turn, an extension of $K$-group schemes
\begin{equation}\label{eqn:exa}
\CD
1 @>>> U @>>> G_K @>{\alpha}>> A_K @>>> 0 
\endCD
\end{equation}
where $U$ has a filtration with subquotients isomorphic to the Lie 
algebra of $A_K$. In particular, $U$ is smooth, connected and
unipotent. Moreover, $\alpha$ is the Albanese morphism of
$(G_K,e_{G_K})$.

For any $k$-scheme $S$, the map $G(S) \to A_K(S_K)$ 
that sends any $f : S \to G$ to $\alpha f_K : S_K \to A_K$
is bijective. This yields a morphism $\beta : A \to G$ such that 
$\alpha \beta_K$ is the identity map of $A_K$. It follows
that $\beta$ is a closed immersion of group schemes; we shall
identify $A$ with $\beta(A)$, and likewise $A_K$ with
$\beta_K(A_K)$. As $\beta_K$ splits the extension (\ref{eqn:exa}), 
this yields the decomposition (\ref{eqn:prK}).

As a consequence, $G_{\ant} = A$ and hence $G = G_{\aff} \, A$. 
Also, $G_{\aff}$ is not smooth; indeed, any morphism from a 
connected affine algebraic group to $G$ is constant, as follows from
the equality (\ref{eqn:uni}) together with \cite[Lem.~2.3]{Co02}. 
Thus, the finite group scheme $G_{\aff} \cap A$ is non-trivial:
otherwise, $G \simeq G_{\aff} \times A$, so that $G_{\aff}$ 
would be smooth. In particular, 
$(G_{\aff})_K \neq U = (G_{\aff})_K$.

In particular, $G/G_{\aff}$ is the quotient of $A$ by a non-trivial
subgroup scheme. On the other hand, the quotient map 
$\alpha_G : G \to G/G_{\aff}$ is easily seen to be the Albanese
morphism of $(G,e_G)$ considered in \cite{Wit06}. Thus, the formation
of the Albanese morphism does not commute with arbitrary field
extensions.  
\end{example}

\subsection{Structure of connected commutative algebraic groups}
\label{subsec:com}

We first obtain a simple characterization of non-affine group 
schemes that are minimal for this property:

\begin{proposition}\label{prop:min}
The following conditions are equivalent for a non-trivial
group scheme $G$:

\smallskip

\noindent
{\rm (i)} $G$ is non-affine and every subgroup scheme 
$H \subset G$, $H \neq G$ is affine.

\smallskip

\noindent
{\rm (ii)}
$G$ is anti-affine and has no non-trivial anti-affine
subgroup.

\smallskip

\noindent
{\rm (iii)} $G$ is anti-affine and the abelian variety 
$A(G) = G/G_{\aff}$ is simple.

\smallskip

If one of these conditions holds, then either $G$ is an abelian 
variety or $G$ contains no complete subvariety of positive dimension.
\end{proposition}

\begin{proof}
(i)$\Leftrightarrow$(ii) follows easily from the fact that a 
group scheme $H$ is affine if and only if $H^0_{\ant}$ is trivial.

(ii)$\Rightarrow$(iii) Assume that $A(G)$ contains a non-trivial 
abelian variety $B$, and denote by $H$ the pull-back of $B$ in $G$. 
Then $H_{\ant}$ is a non-trivial subgroup of $G$, a contradiction.

(iii)$\Rightarrow$(ii) Let $H$ be an anti-affine subgroup of $G$.
Then $H_{\aff} \subset G_{\aff}$, as $G_{\aff}$ is the largest connected 
affine subgroup of $G$; hence $A(H)$ is identified with a subgroup of 
$A(G)$. Thus, either $A(H)$ is trivial so that $H$ is affine, or 
$A(H) = A(G)$ so that $G_{\aff}\, H = G$. In the latter case, $H = G$ 
by Proposition \ref{prop:ros}.

Under one of these conditions, consider the algebraic subgroup 
$H \subset G$ generated by a complete subvariety of $G$. Then $H$ 
is complete as well (see e.g. \cite[Exp.~VIB, Prop.~7.1]{SGA3});
thus, either $H = G$ or $H$ is trivial.
\end{proof}

Next, we obtain a decomposition of connected commutative group schemes
over perfect fields:

\begin{theorem}\label{thm:com}
Let $G$ be a connected commutative group scheme over a perfect field 
$k$. Then there exist a subtorus $T \subset G$ and a connected unipotent
subgroup scheme $U \subset G$ such that the group law of $G$ induces an 
isogeny
\begin{equation}\label{eqn:com}
f : G_{\ant} \times T \times U \longrightarrow G.
\end{equation}
Moreover, $T$ is unique up to isogeny, and $U$ is unique up to 
isomorphism; if $G$ is an algebraic group, then so is $U$.
\end{theorem}

\begin{proof}
The Rosenlicht decomposition yields an exact sequence of group schemes
$$
\CD
1 @>>> G_{\aff} \cap G_{\ant} @>>> G_{\aff} @>{\psi}>> G/G_{\ant} 
@>>> 1.
\endCD
$$
Moreover, we have unique decompositions $G_{\aff} = T' \times U'$
and $G/G_{\ant} = T'' \times U''$, where $T'$, $T''$ are tori and $U'$, 
$U''$ are connected unipotent group schemes. This yields epimorphisms
$\psi_s : T' \to T''$, $\psi_u : U' \to U''$. Thus, we may find
a subtorus $T \subset T'$ such that $\psi_s$ restricts to an isogeny 
$T \to T''$. 

If $k$ has characteristic $0$, we may also find a (connected) unipotent
subgroup $U \subset U'$ such that $\psi_u$ restricts to an 
isomorphism $U \to U''$, as $U'$ and $U''$ are vector groups. 
Then the homomorphism $f$ induces an isogeny 
$T \times U \to G/G_{\ant}$. Thus, $f$ is an isogeny, and $T$, $U$ are 
unique up to isogeny; hence the vector group $U$ is uniquely determined. 

In positive characteristics, $G_{\aff} \cap G_{\ant}$ contains the 
torus $(G_{\ant})_{\aff}$ and the quotient is finite; hence $\psi_u$ 
is an isogeny. Thus, our statement holds with $U = U'$, but for no other 
choice of $U$. 
\end{proof}

The assumption that $k$ is perfect cannot be omitted in the preceding 
result, as shown by Example \ref{ex:npe}.

\subsection{Further decompositions in positive characteristics}
\label{subec:dec}

In this subsection, we combine the Rosenlicht decomposition 
with the particularly simple structure of anti-affine algebraic 
groups in positive characteristics, to obtain information on 
general algebraic groups.

We begin with the case where the field $k$ is finite. Then 
Propositions \ref{prop:pos} and \ref{prop:ros} immediately imply
the following result, due to Arima in the setting of algebraic 
groups (see \cite[Thm.~1]{Ar60} and also \cite[Thm.~4]{Ro61}):

\begin{proposition}\label{prop:ari}
Let $G$ be a connected group scheme over a finite field $k$. Then 
$G = G_{\aff} \, G_{\ab}$ where $G_{\ab}$ denotes the largest 
abelian subvariety of $G$. Moreover, $G_{\aff} \cap G_{\ab}$ is finite.
\end{proposition}

In particular, the Albanese morphism of $(G,e_G)$ is trivialized by
the finite cover $G_{\aff} \times G_{\ab} \to G$ (possibly
non-\'etale).

Returning to a possibly infinite field $k$, we record the following 
preliminary result:

\begin{lemma}\label{lem:gal}
Let $G$ be a connected algebraic group over a perfect field $k$. Then:

\smallskip

\noindent
{\rm (i)} There exists a smallest normal connected algebraic group 
$H \subset G$ such that $G/H$ is a semi-abelian variety.
The quotient homomorphism $G \to G/H$ is the generalized
Albanese morphism of the pointed variety $(G,e_G)$.

\smallskip

\noindent
{\rm (ii)} We have
\begin{equation}\label{eqn:der}
H = R_u(G_{\aff}) \, [G,G] = R_u(G_{\aff}) \, [G_{\aff},G_{\aff}]
\end{equation}
where $R_u(G_{\aff})$ denotes the unipotent radical of $G_{\aff}$, and 
$[G,G]$ the derived group. 

\smallskip

\noindent
{\rm (iii)} The formation of $H$ commutes with perfect field
extensions.

\smallskip

\noindent
{\rm (iv)} The group $H_{\kb}$ is generated by all connected unipotent 
subgroups of~$G_{\kb}$.
\end{lemma}

\begin{proof}
By the Rosenlicht decomposition, we have $[G,G] = [G_{\aff},G_{\aff}]$.
Define $H$ by the equality (\ref{eqn:der}); then $H$ 
is a connected normal subgroup of $G$. Moreover, the quotient 
$G_{\aff}/H$ is a connected commutative reductive group, 
i.e., a torus. Thus, $G/H$ is a semi-abelian variety. 

Consider a morphism $f : G \to S$, where $S$ is a semi-abelian
variety, and $f(e_S) = e_G$. Then $f$ is a homomorphism by
\cite[Thm.~3]{Ro61}. Hence $f$ factors through $G/R_u(G_{\aff})$ 
(as every unipotent subgroup of $S$ is trivial) and also through 
$G/[G,G]$ (as $S$ is commutative). Thus, $f$ factors through
$G/H$. This proves (i) and (ii), while (iii) and (iv) are obtained by 
similar arguments.
\end{proof}

Under the assumptions of the preceding lemma, we say that $H$
is \emph{geometrically unipotently generated}, and write 
$H: = G_{\gug}$.

Also, given a group scheme $G$, a normal subgroup scheme  
$H \subset G$ and a subgroup scheme $S \subset G$, we say that 
$S$ is a \emph{quasi-complement to $H$ in $G$} if $G = H \, S$ and 
$H \cap S$ is finite; equivalently, the natural map
$S \to G/H$ is an isogeny. We may now state our structure result:

\begin{theorem}\label{thm:dec}
Let $G$ be a connected algebraic group over a perfect field $k$ 
of positive characteristic and let $T$ be a maximal torus of the
radical $R(G_{\aff})$. Then:

\smallskip

\noindent
{\rm (i)} $T$ is a quasi-complement to $G_{\gug}$ in $G_{\aff}$. 

\smallskip

\noindent
{\rm (ii)} $S := T \, G_{\ant}$ is a quasi-complement to $G_{\gug}$ in
$G$, and is a semi-abelian subvariety of $G$ with maximal torus $T$. 

\smallskip

\noindent
{\rm (iii)} The generalized Albanese morphism of $(G,e_G)$ is
trivialized by the finite cover $G_{\gug} \times S \to G$.
\end{theorem}

\begin{proof}
(i) By the structure of affine algebraic groups (see \cite{Bo91})
and the equality (\ref{eqn:der}), we have
$$\displaylines{
G_{\aff}  = R(G_{\aff}) \, [G_{\aff}, G_{\aff}] 
= R_u(G_{\aff}) \, T \, [G_{\aff}, G_{\aff}] 
\hfill \cr \hfill
= R_u(G_{\aff}) \, [G_{\aff}, G_{\aff}] \, T
= G_{\gug} \, T.
\cr}$$
We now show the finiteness of $G_{\gug} \cap T$. For this, we may
assume that $G$ is affine. Since the homomorphism 
$$
G_{\gug} \cap T \to (G_{\gug} \cap R(G))/R_u(G) 
\subset G/R_u(G)
$$ 
is finite, we may also assume $G$ to be reductive. Then 
$G_{\gug} = [G,G]$ is semi-simple and $T$ is the largest central
torus, so that their intersection is indeed finite.

(ii) By the Rosenlicht decomposition and (i), 
$G = G_{\gug} \, S$. Also, the quotient 
$(G_{\gug} \cap S)/(G_{\gug} \cap T)$ is finite, since 
$G_{\gug} \cap S$ is affine and $T = S_{\aff}$. Thus, 
$G_{\gug} \cap S$ is finite, i.e., $S$ is a quasi-complement to
$G_{\gug}$ in $G$. 

We know that $G_{\ant}$ is a semi-abelian variety contained in 
the centre of $G$. Thus, $S$ is a semi-abelian variety as well. 
Moreover, the maximal torus $(G_{\ant})_{\aff}$ of $G_{\ant}$
is a central subtorus of $G_{\aff}$, and hence is contained in $T$.
Thus, $T$ is the maximal torus of $S$.

(iii) follows readily from (ii).
\end{proof}

\begin{remarks}
(i) The quasi-complements constructed in the preceding theorem are all
conjugate under $R_u(G_{\aff})$. But $G_{\gug}$ may admit other
quasi-complements in $G$; namely, all subgroups  $T' \, G_{\ant}$ where
$T'$ is a quasi-complement of $G_{\gug}$ in $G_{\aff}$. Such a
subtorus $T'$ need not be contained in $R(G_{\aff})$, e.g., when
$G_{\aff}$ is reductive and non-commutative.

\smallskip

\noindent
(ii) With the notation and assumptions of the preceding theorem, 
$G_{\ant}$ also admits quasi-complements in $G$, namely, the subgroups 
$T' \,G_{\gug}$ where $T'$ is a quasi-complement to
$(G_{\ant})_{\aff}$ in $T$.

In contrast, $G_{\aff}$ may admit no quasi-complement in $G$. Indeed,
such a quasi-complement $S$ comes with a finite surjective morphism to
$G/G_{\aff}$, and hence is an abelian variety. Thus, $S$ exists if and
only if $G_{\ant}$ is an abelian variety, and then $S = G_{\ant}$. 
Equivalently, $G = G_{\aff} \, G_{\ab}$ as in Proposition \ref{prop:ari}.

\smallskip

\noindent
(iii) In characteristic $0$, the group $G_{\gug}$ still admits
quasi-complements in $G_{\aff}$, but may admit no quasi-complement
in $G$.

For example, let $C$ be an elliptic curve, $E(C)$ its universal
extension, $H$ the Heisenberg group of upper triangular $3 \times 3$
matrices with diagonal entries $1$, and $G = (H \times E(C))/\bG_a$ where
the additive group $\bG_a$ is embedded in $H$ as the center, and in
$E(C)$ as $E(C)_{\aff}$. Then $G$ is a connected algebraic group;
moreover, $G_{\gug} = G_{\aff} \cong  H$ and 
$G_{\ant} \cong E(C)$. Since $G_{\ant}$ is non-complete, there exist
no quasi-complement to $G_{\aff}$ in $G$.

The same example shows that $G_{\ant}$ may admit no quasi-complement
in $G$. Yet such a quasi-complement does
exist when $G$ is commutative, by Theorem \ref{thm:com}.
\end{remarks}

\subsection{Counterexamples to Hilbert's fourteenth problem}
\label{subsec:hil}

In this subsection, we construct a class of smooth 
quasi-affine varieties having a non-noetherian coordinate ring. 

Recall that every connected algebraic group $G$ is quasi-projective, 
i.e., $G$ admits an ample invertible sheaf $\cL$ (see e.g. 
\cite[Cor.~V 3.14]{Ra70}). Clearly, the associated $\bG_m$-torsor 
over $G$ (that is, the complement of the zero section in the total 
space of the associated line bundle $\bV(\cL)$) is a smooth
quasi-affine variety. This simple construction yields our examples:

\begin{theorem}\label{thm:zar}
Let $\pi : X \to G$ denote the $\bG_m$-torsor associated to an ample 
invertible sheaf $\cL$ on a non-complete anti-affine algebraic group. 
Then the ring $\cO(X)$ is not noetherian.
\end{theorem}

\begin{proof}
As $X = \Spec_{\cO_G}( \bigoplus_{n \in \bZ} \cL^n)$, we have
$\cO(X) = \bigoplus_{n \in \bZ} H^0(G, \cL^n)$.
Moreover, $H^0(G,\cO_G) = k$ by assumption, and the $k$-vector 
space $H^0(G, \cL^n)$ is infinite-dimensional for any $n > 0$ 
by the next lemma. Since $\cO(X)$ is a domain, it follows that 
$H^0(G, \cL^n) = 0$ for any $n < 0$, i.e., the algebra $\cO(X)$ is 
positively graded. Clearly, this algebra is not finitely generated, 
and hence non-noetherian by the graded version of Nakayama's lemma. 
\end{proof}

\begin{lemma}\label{lem:inf}
Let $\cL$ be an ample invertible sheaf on an anti-affine 
algebraic group $G$. If $G$ is non-complete, then the 
$k$-vector space $H^0(G,\cL)$ is infinite-dimensional.
\end{lemma}

\begin{proof}
We may assume that $k$ is algebraically closed.
The quotient homomorphism $\alpha = \alpha_G : G \to A(G) =: A$ 
is a torsor under the connected commutative affine algebraic group 
$G_{\aff}$. Since the Picard group of $G_{\aff}$ is trivial, 
it follows that $\cL = \alpha^*(\cM)$ for some invertible sheaf 
$\cM$ on $A$. Moreover, $\cM$ is ample by the ampleness of $\cL$ 
together with \cite[Lem.~XI 1.11.1]{Ra70}. We have 
\begin{equation}\label{eqn:dir}
H^0(G,\cL) \simeq H^0 \big( A, \cM \otimes \alpha_* (\cO_G) \big).
\end{equation}

In the case where $G$ is a semi-abelian variety, Equations
(\ref{eqn:dec}) and (\ref{eqn:dir}) yield the decomposition
$$
H^0(G,\cL) \simeq \bigoplus_{\lambda \in \Lambda}
H^0(A, \cM \otimes \cL_{\lambda}).
$$
As each $\cL_{\lambda}$ is algebraically
trivial, $\cM \otimes \cL_{\lambda}$ is ample, and hence 
admits non-zero global sections (see \cite[p.~163]{Mu70}); this 
yields our statement in this case.

In the general case, we may assume in view of Proposition 
\ref{prop:pos} and the isomorphism (\ref{eqn:tens}) that 
$k$ has characteristic $0$, and $G_{\aff}$ is a non-zero vector 
space $U$. Then $\cM \otimes \alpha_* (\cO_G)$ admits an infinite 
increasing filtration with subquotients isomorphic to $\cM$, 
by Lemma \ref{lem:fib}. Since $H^1(A,\cM) = 0$ 
(see \cite[p.~150]{Mu70}), it follows that $H^0(G,\cL)$ admits 
an infinite increasing filtration with subquotients isomorphic 
to $H^0(A,\cM)$, a non-zero vector space.
\end{proof}

\begin{example}\label{ex:ell}
The smallest examples arising from the preceding construction are
threefolds; they may be described as follows.

Consider an invertible sheaf $\cL$ of positive degree on an elliptic 
curve $C$. If $k$ has characteristic $0$, let $\pi : G \to C$ 
denote the $\bG_a$-torsor associated to the canonical generator of 
$H^1(C,\cO_C) \simeq H^0(C, \cO_C)^*$. Then $G$ is the universal
extension $E(C)$, and the $\bG_m$-torsor on $G$ associated to the
ample invertible sheaf $\pi^*(\cL)$ yields the desired example $X$.

When $k = \bC$, the analytic manifolds associated to $G$ and $X$ 
are both Stein; see \cite{Ne88} which also contains an analytic 
proof of the fact that $\cO(X)$ is not finitely generated. 
More generally, the universal extension $E(A)$ of a complex abelian
variety of dimension $g$ is analytically isomorphic to $(\bC^*)^{2g}$, 
see e.g. \cite[Rem.~7.7]{Ne88}. In particular, the complex manifold 
associated to $E(A)$ is Stein.

Returning to a field $k$ of arbitrary characteristics, assume that the
elliptic curve $C$ has a $k$-rational point $x$ of infinite order
(such curves exist if $k$ contains either $\bQ$ or $\bF_p(t)$, see
\cite{ST67}). Denote by $\cM$ the invertible sheaf on $C$ associated
to the divisor $(x) - (0)$. Then $\cM$ is algebraically trivial and
has infinite order. Thus, 
$G := \Spec_{\cO_C} (\bigoplus_{n \in \bZ} \cM^n)$
is an anti-affine semi-abelian variety, and 
$$
X := \Spec_{\cO_C} \big( \bigoplus_{(m, n) \in \bZ^2} 
\cL^m \otimes_{\cO_A} \cM^n \big)
$$
is the desired example.

It should be noted that $\cO(X)$ is finitely generated for any smooth
surface $X$, as shown by Zariski (see \cite{Za54}). Also, Kuroda has
constructed counterexamples to Hilbert's original problem, in dimension
$3$ and characteristic $0$ (see \cite{Ku05}). 
\end{example}

Another consequence of Lemma \ref{lem:inf} is the following:

\begin{proposition}\label{prop:bou}
For any completion $\overline{G}$ of a connected algebraic group 
$G$, the boundary $\overline{G} \setminus G$ is either empty or 
of codimension $1$.
\end{proposition}

\begin{proof}
We argue by contradiction, and assume that 
$\overline{G} \setminus G$ is non-empty of codimension $\geq 2$. 
We may further assume that $\overline{G}$ is normal; then the 
map $i^*: \cO(\overline{G}) \to \cO(G)$
is an isomorphism, where $i : G \to \overline{G}$ denotes the
inclusion. It follows that $G$ is anti-affine and non-complete.

Choose an ample invertible sheaf $\cL$ on $G$. Then $i_*(\cL)$ is the
sheaf of sections of some Weil divisor on $\overline{G}$; in
particular, this sheaf is coherent. Thus, the $k$-vector space
$H^0(\overline{G}, i_*(\cL)) = H^0(G,\cL)$ 
is finite-dimensional, contradicting Lemma \ref{lem:inf}.
\end{proof}

\begin{remark}
With the assumptions of the preceding proposition, one may show 
(by completely different methods) that the boundary has \emph{pure}
codimension $1$. For a $G$-\emph{equivariant} completion
$\overline{G}$ (that is, the action of $G$ on itself by left
multiplication extends to $\overline{G}$), this follows easily from  
\cite[Thm.~3]{Br07}. Namely, we may assume that $k$ is 
algebraically closed and $\overline{G}$ is normal; then 
$\overline{G} \simeq G \times^{G_{\aff}} \overline{G_{\aff}}$, and 
$\overline{G_{\aff}} \setminus G_{\aff}$ has pure codimension $1$ 
in $\overline{G_{\aff}}$, as $G_{\aff}$ is affine.
\end{remark}

\end{document}